\documentclass[12pt]{article}
\usepackage{amssymb}
\usepackage{mathrsfs}
\usepackage[hypertex]{hyperref}
\usepackage{amsfonts}
\usepackage{graphicx}
\usepackage{latexsym,amsmath,amssymb,amsfonts,amsthm}

\newcommand{\bcen}{\begin{center}}
\newcommand{\ecen}{\end{center}}
\newtheorem{theorem}{Theorem}[section]
\newtheorem{lemma}[theorem]{Lemma}

\newtheorem{corollary}[theorem]{Corollary}

\newtheorem{remark}[theorem]{Remark}

\def\L{\mathbb{L_{\phi}}}

\begin{document}
\setcounter{page}{1}
\title{Brock-type isoperimetric inequality for Steklov eigenvalues of the Witten-Laplacian}
\author{Jing Mao$^{\ast}$,~~Shijie Zhang}

\date{}
\protect \footnotetext{\!\!\!\!\!\!\!\!\!\!\!\!{~$\ast$ Corresponding author} }
\maketitle ~~~\\[-15mm]

\begin{center}
{\footnotesize Faculty of Mathematics and Statistics,\\
 Key Laboratory of Applied
Mathematics of Hubei Province, \\
Hubei University, Wuhan 430062, China\\
Email: jiner120@163.com
 }
\end{center}


\begin{abstract}
In this paper, by imposing suitable assumptions on the weighted
function, (under the constraint of fixed weighted volume) a
Brock-type isoperimetric inequality for Steklov-type eigenvalues of
the Witten-Laplacian on bounded domains in a Euclidean space or a
hyperbolic space has been proven. This conclusion is actually an
interesting extension of F. Brock's classical result about the
isoperimetric inequality for Steklov eigenvalues of the Laplacian
given in the influential paper [Z. Angew. Math. Mech. {\bf 81}
(2001) 69--71]. Besides, a related open problem has also been
proposed in this paper.
 \end{abstract}


\section{Introduction}
\renewcommand{\thesection}{\arabic{section}}
\renewcommand{\theequation}{\thesection.\arabic{equation}}
\setcounter{equation}{0}

Let $(M^{n},\langle\cdot,\cdot\rangle)$ be an $n$-dimensional
($n\geq2$) complete Riemannian manifold\footnote{~Without
specifications, generally, in this paper same symbols have the same
meanings.} with the metric $g:=\langle\cdot,\cdot\rangle$. Let
$\Omega\subset M^{n}$ be a
 domain in $M^n$, and $\phi\in C^{\infty}(M^n)$ be a
smooth\footnote{~In fact, one might see that $\phi\in C^{2}$ is
suitable to derive our main conclusions in this paper. However, in
order to avoid a little bit boring discussion on the regularity of
$\phi$ and following the assumption on conformal factor $e^{-\phi}$
for the notion of \emph{smooth metric measure spaces} in many
literatures (including of course those cited in this paper), without
specification, we prefer to assume that $\phi$ is smooth on the
domain $\Omega$. } real-valued function defined on $M^n$. In this
setting, on $\Omega$, the following elliptic operator\footnote{~In
some previous works of the corresponding author, the symbol for the
Witten-Laplacian is $\Delta_{\phi}$, but in some others it is given
as $\L$. Since we (\cite[Appendix A]{DDMZ}) firstly and formally
give the detailed proof for the existence of discrete spectrum of
the eigenvalue problem (\ref{a10}) considered below, following the
usage of the symbol for the Witten-Laplacian in \cite{DDMZ}, we
prefer to also use the symbol $\L$ to denote the Witten-Laplacian in
this paper.}
\begin{eqnarray*}
\L:=\Delta-\langle\nabla\phi,\nabla\cdot\rangle
\end{eqnarray*}
can be well-defined, where $\nabla$, $\Delta$ are the gradient and
the Laplace operators on $M^{n}$, respectively. The operator $\L$
w.r.t. the metric $g$ is called the \emph{Witten-Laplacian} (also
called the \emph{drifting Laplacian} or the \emph{weighted
Laplacian}). It has historical background, realistic needs and
theoretical significance to study this operator no matter in
spectral analysis on manifolds but also in some interesting and hot
research topics in Geometric Analysis. To have a nice experience
about this statement, we refer readers to some of the corresponding
author's previous papers \cite[Section 1]{CM}, \cite[Section 1]{CM1}
and \cite[Section 1 and Remark 4.3]{DMWW}, where the reason and the
motivation of investigating the Witten-Laplacian $\L$ have been
explained clearly.

Using the conformal measure $d\mu:=e^{-\phi}dv$, the notion,
\emph{smooth metric measure space} $(M^{n},g,d\mu)$, can be
well-defined, which is actually the given Riemannian manifold
$(M^{n},g)$ equipped with the weighted measure $d\mu$. Smooth metric
measure space $(M^{n},g,d\mu)$ sometimes is also called the
\emph{weighted measure space}. For the smooth metric measure space
$(M^{n},g,d\mu)$, one can define a notion, weighted volume (or
$\phi$-volume), as follows:
\begin{eqnarray*}
|M^{n}|_{\phi}:=\int_{M^{n}}d\mu=\int_{M^n}e^{-\phi}dv,
\end{eqnarray*}
 where $dv$ is the Riemannian volume element (or volume density)
 of $M^{n}$, and sometimes $d\mu$ is also called the weighted volume
 density (or weighted measure) w.r.t. the metric $g$.
On a compact smooth metric measure space
$(\Omega,\langle\cdot,\cdot\rangle,d\mu)$, one can naturally
consider the Steklov-type eigenvalue problem of the Witten-Laplacian
$\L$ as follows
\begin{eqnarray} \label{a10}
\left\{\begin{array}{ll} \L u
=0&~~\mathrm{in} ~~\Omega\subset M^{n}, \\[2mm]
\frac{\partial u}{\partial \vec{\eta}}=\sigma
u&~~\mathrm{on}~~\partial \Omega,
\end{array}\right.
\end{eqnarray}
where $\vec{\eta}$ is the outward unit  normal vector along the
Lipschitz continuous boundary $\partial\Omega$. By \cite[Appendix
A]{DDMZ}, one knows that the operator $\L$  in (\ref{a10}) has
discrete spectrum and its relation with a Dirichlet-to-Neumann
operator can be described as follows:\footnote{~As explained in the
footnote of the $16^{\mathrm{th}}$ page of \cite{DDMZ}, here a
convention has been used, that is, for the Sobolev space
$\widetilde{W}^{k,p}(\cdot)$ w.r.t. the weighted volume density, if
$p=2$, we usually write
$\widetilde{H}^{k}(\cdot)=\widetilde{W}^{k,p}(\cdot)$.}
\begin{itemize}
\item \emph{ The eigenvalues of the eigenvalue problem of the
Witten-Laplacian (\ref{a10})  can be interpreted as the eigenvalues
of the Dirichlet-to-Neumann operator
$\mathcal{D}:\widetilde{H}^{1/2}(\partial\Omega)\rightarrow
\widetilde{H}^{-1/2}(\partial\Omega)$ which maps a function
$f\in\widetilde{H}^{1/2}(\partial\Omega)$ to
$\mathcal{D}f=\partial_{\vec{\eta}}(H_{w}f)=\frac{\partial
(H_{w}f)}{\partial\vec{\eta}}$, where $H_{w}f$ is the weighted
harmonic extension of $f$ to the interior of $\Omega$ (i.e.,
$\L(H_{w}f)=0$ in $\Omega$).}
\end{itemize}
All the elements (i.e., eigenvalues) in the discrete spectrum of the
eigenvalue problem (\ref{a10}) can be listed non-decreasingly as
follows
 \begin{eqnarray} \label{sequence-1}
0=\sigma_{0,\phi}(\Omega)<\sigma_{1,\phi}(\Omega)\leq\sigma_{2,\phi}(\Omega)\leq\sigma_{3,\phi}(\Omega)\leq\cdots\uparrow\infty.
 \end{eqnarray}
 Each eigenvalue $\sigma_{i,\phi}(\Omega)$, $i=0,1,2,\cdots$, in the sequence (\ref{sequence-1}) is repeated
 according to its multiplicity (i.e. the dimension of the eigenspace of the eigenvalue $\sigma_{i,\phi}(\Omega)$, which is
 finite). Clearly, an arbitrary nonzero constant function can be chosen as an
 eigenfunction of the first eigenvalue $\sigma_{0,\phi}(\Omega)=0$.
 By the variational principle, the $k$-th nonzero eigenvalue
 $\sigma_{k,\phi}(\Omega)$ can be characterized as follows
\begin{eqnarray} \label{cha-1}
\sigma_{k,\phi}(\Omega)=\min\left\{\frac{\int_{\Omega}|\nabla
u|^{2}d\mu}{\int_{\partial\Omega}u^{2}\widehat{d\mu}}\Bigg{|}u\in
\widetilde{W}^{1,2}(\Omega),u\neq0,\int_{\partial\Omega}uu_{i}\widehat{d\mu}=0\right\},
 \end{eqnarray}
where $u_{i}$, $i=0,1,2,\cdots,k-1$, denotes an eigenfunction of
$\sigma_{i,\phi}(\Omega)$. Here let us make an explanation on the
meaning of two notations in (\ref{cha-1}). First, following the
convention in \cite{CM}, denote by $\widehat{dv}$ the
$(n-1)$-dimensional Hausdorff measure of the boundary associated to
the Riemannian volume element $dv$, and this usage will be used
throughout the paper. Similarly,
$\widehat{d\mu}=e^{-\phi}\widehat{dv}$ would be the weighted volume
element of the boundary. Second, as explained in \cite[Section
1]{DDMZ}, $\widetilde{W}^{1,2}(\Omega)$ should be the completion of
the set of smooth functions $C^{\infty}(\Omega)$ under the weighted
Sobolev norm
${\widetilde{\|u\|}}_{1,2}:=\left(\int_{\Omega}u^{2}d\mu+\int_{\partial\Omega}|\nabla
u|^{2}\widehat{d\mu}\right)^{1/2}$. Specially, the first nonzero
eigenvalue $\sigma_{1,\phi}(\Omega)$ of the eigenvalue problem
(\ref{a10}) satisfies
\begin{eqnarray} \label{cha-2}
\sigma_{1,\phi}(\Omega)=\min\left\{\frac{\int_{\Omega}|\nabla
u|^{2}d\mu}{\int_{\partial\Omega}u^{2}\widehat{d\mu}}\Bigg{|}u\in
\widetilde{W}^{1,2}(\Omega),u\neq0,\int_{\partial\Omega}u\widehat{d\mu}=0\right\}.
 \end{eqnarray}
For convenience and without confusion, in the sequel, except
specification we will write $\sigma_{i,\phi}(\Omega)$ as
$\sigma_{i,\phi}$ directly. This convention would be also used when
we meet with other possible eigenvalue problems.

Before giving the main conclusions of this paper, we prefer to
briefly show the connection and also the difference between the
eigenvalue problem (\ref{a10}) and the classical Steklov eigavalue
problem of the Laplacian. In fact, if $\phi=const.$ is a constant
function, then $\L$ degenerates into the classical Laplace operator
$\Delta$, and correspondingly  (\ref{a10}) becomes
\begin{eqnarray} \label{a11}
\left\{\begin{array}{ll} \Delta u
=0&~~\mathrm{in} ~~\Omega\subset M^{n}, \\[2mm]
\frac{\partial u}{\partial \vec{\eta}}=\sigma
u&~~\mathrm{on}~~\partial \Omega.
\end{array}\right.
\end{eqnarray}
For bounded domains in the Euclidean plane $\mathbb{R}^2$, this
problem was firstly introduced by M. W. Stekloff \cite{MWS} over 120
years ago, and his motivation came from physics -- the function $u$
represents the steady state temperature on $M^{n}$ such that the
flux on the boundary is proportional to the temperature. Based on
this reason, (\ref{a11}) is formally called \emph{the Steklov
eigenvalue problem} of the Laplacian. Since the Steklov eigenvalue
problem (\ref{a11}) was proposed, many mathematicians focused on
this topic and progresses have been made continuously (especially
recent years). For a more detailed introduction on the historical
background and physical motivations of the problem (\ref{a11}),
please see a memorial article \cite{KKKNP}. Readers are invited to
check two survey articles \cite{CGGS,GP} (especially the longer one,
161 pages, published very recently) to know some recent developments
on the Steklov eigenvalue problem and to try to see the panorama of
this problem and also the relations with other topics in
Differential Geometry.

One might have an illusion that since the eigenvalue problem
(\ref{a10}) can be obtained from the Steklov eigenvalue problem by
replacing the Laplacian $\Delta$ by the Witten-Laplacian $\L$, it
might be easy to derive conclusions for eigenvalues of the problem
(\ref{a10}) directly from the Steklov eigenvalue problem
(\ref{a11}). Readers would see that this thought is a little bit
naive through only a simple example. In fact, for the Steklov
eigenvalue problem (\ref{a11}), if $\Omega$ is a unit disk in the
plane $\mathbb{R}^2$, one can directly compute its Steklov
eigenvalues and all the eigenvalues can be listed non-decreasingly
as follows
\begin{eqnarray*}
0,1,1,2,2,\cdots,k,k,\cdots,
\end{eqnarray*}
whose corresponding eigenfunctions in polar coordinates $(r,\theta)$
are given by
\begin{eqnarray*}
1,r\sin\theta,r\cos\theta,r\sin2\theta,r\cos2\theta,\cdots,r\sin
k\theta,r\cos k\theta,\cdots.
\end{eqnarray*}
However, for the eigenvalue problem (\ref{a10}) and under the
assumption that $\Omega$ is a unit disk in $\mathbb{R}^2$, (except
the trivial case $\phi=const.$) it is impossible to compute
\emph{explicitly} all the eigenvalues of the Witten-Laplacian even
if $\phi$ is a radial function. Hence, when jumping from the
classical Steklov eigenvalue problem (\ref{a11}) to the Steklov-type
eigenvalue problem (\ref{a10}) of the Witten-Laplacian, one do has
some difficulties to overcome. Besides, since  (\ref{a10}) has
relation with  (\ref{a11}) formally (especially their boundary
conditions have almost the same type), we prefer to call the
eigenvalues of the eigenvalue problem (\ref{a10}) the
\emph{Steklov-type eigenvalues of the Witten-Laplacian} (or
sometimes, simply, the Steklov eigenvalues of the Witten-Laplacian).

 In this paper, we focus on the Steklov-type eigenvalue problem
(\ref{a10}) of the Witten-Laplacian and can prove an isoperimetric
inequality for the sums of the reciprocals of the first $(n-1)$
nonzero Steklov eigenvalues of the Witten-Laplacian on bounded
domains in $\mathbb{R}^n$ or a hyperbolic space. However, in order
to state our conclusions clearly, we need to impose an assumption on
the function $\phi$ as follows:
\begin{itemize}
\item (\textbf{Property I}) Furthermore, $\phi$ is a function
of the Riemannian distance parameter $t:=d(o,\cdot)$ for some point
$o\in\mathrm{hull}(\Omega)$, and $\phi$ is also a non-increasing
concave function defined on $[0,\infty)$.
\end{itemize}
Here $\mathrm{hull}(\Omega)$ stands for the convex hull of the
domain $\Omega$. Clearly, if a given open Riemannian $n$-manifold
$(M^{n},g)$ was endowed with the weighted density $e^{-\phi}dv$ with
$\phi$ satisfying \textbf{Property I}, then $\phi$ would be a
\emph{\textbf{radial}} function defined on $M^{n}$ w.r.t. the radial
distance $t$, $t\in[0,\infty)$. Especially, when the given open
$n$-manifold is chosen to be $\mathbb{R}^{n}$ or $\mathbb{H}^{n}$
(i.e., the $n$-dimensional hyperbolic space of sectional curvature
$-1$), we additionally require that $o$ is the origin  of
$\mathbb{R}^{n}$ or $\mathbb{H}^{n}$.

\begin{theorem} \label{theo-1}
Assume that the function $\phi$ satisfies \textbf{Property I}. Let
$\Omega$ be a bounded domain with Lipschitz continuous boundary in
 $\mathbb{R}^n$, and let $B_{R}(o)$ be a ball of radius $R$ and centered at the
 origin $o$
 of
 $\mathbb{R}^{n}$ such that $|\Omega|_{\phi}=|B_{R}(o)|_{\phi}$,
 i.e. $\int_{\Omega}d\eta=\int_{B_{R}(o)}d\eta$.  Then
  \begin{eqnarray} \label{II-1}
\frac{1}{\sigma_{1,\phi}(\Omega)}+\frac{1}{\sigma_{2,\phi}(\Omega)}+\cdots+\frac{1}{\sigma_{n-1,\phi}(\Omega)}\geq\frac{n-1}{\sigma_{1,\phi}(B_{R}(o))}.
\end{eqnarray}
The equality case holds if and only if $\Omega$ is the ball
$B_{R}(o)$.
\end{theorem}

By applying the sequence (\ref{sequence-1}), i.e. the monotonicity
of Steklov eigenvalues of the Witten-Laplacian, from (\ref{II-1})
one has:

\begin{corollary} \label{coro-1}
Under the assumptions of Theorem \ref{theo-1}, we have
\begin{eqnarray} \label{II-2}
\sigma_{1,\phi}(\Omega)\leq\sigma_{1,\phi}(B_{R}(o)),
\end{eqnarray}
with equality holding if and only if $\Omega$ is the ball
$B_{R}(o)$. That is to say, among all bounded domains in
$\mathbb{R}^n$ having the same weighted volume, the ball $B_{R}(o)$
maximizes the first nonzero Steklov eigenvalue of the
Witten-Laplacian, provided the function $\phi$ satisfies
\textbf{Property I}.
\end{corollary}

\begin{theorem}  \label{theo-2}
Assume that the function $\phi$ satisfies \textbf{Property I}. Let
$\Omega$ be a bounded domain in $\mathbb{H}^n$, and let $B_{R}(o)$
be a geodesic ball of radius $R$ and centered at the
 origin $o$
 of
 $\mathbb{H}^{n}$ such that $|\Omega|_{\phi}=|B_{R}(o)|_{\phi}$.
 Then
 \begin{eqnarray} \label{II-1-1}
\frac{1}{\sigma_{1,\phi}(\Omega)}+\frac{1}{\sigma_{2,\phi}(\Omega)}+\cdots+\frac{1}{\sigma_{n-1,\phi}(\Omega)}\geq\frac{n-1}{\sigma_{1,\phi}(B_{R}(o))}.
\end{eqnarray}
The equality case holds if and only if $\Omega$ is isometric to the
geodesic ball $B_{R}(o)$.
\end{theorem}

Similarly, by applying the sequence (\ref{sequence-1}),  from
(\ref{II-1-1}) one has:

\begin{corollary}
Under the assumptions of Theorem \ref{theo-2}, we have
\begin{eqnarray} \label{II-2-2}
\sigma_{1,\phi}(\Omega)\leq\sigma_{1,\phi}(B_{R}(o)),
\end{eqnarray}
with equality holding if and only if $\Omega$ is isometric to
$B_{R}(o)$. That is to say, among all bounded domains in
$\mathbb{H}^n$ having the same weighted volume, the geodesic ball
$B_{R}(o)$ maximizes the first nonzero Steklov eigenvalue of the
Witten-Laplacian, provided the function $\phi$ satisfies
\textbf{Property I}.
\end{corollary}

\begin{remark}
\rm{ (1) Similar to the explanation in (1) of \cite[Remark
1.4]{CM1}, the Euclidean $n$-space $\mathbb{R}^n$ is two-points
homogenous, so generally it seems like there is no need to point out
the information of the center for the ball $B_{R}(o)$ when
describing the isometry conclusion in Theorem \ref{theo-1}. However,
for the eigenvalue problem (\ref{a10}), by (\ref{cha-1}) one knows
that even on Euclidean balls, the Steklov eigenvalues
$\sigma_{i,\phi}$ also depend on the weighted function $\phi$
(except the situation that $\phi$ is a constant function). This
implies that for Euclidean balls with the same radius but different
centers, they might have different Steklov eigenvalues
$\sigma_{i,\phi}$ since generally the radial function $\phi$ here
has different distributions on different balls. Therefore, we need
to give the information of the center for the ball $B_{R}(o)$ when
we investigate the possible rigidity for the equality case of
(\ref{II-1}). The same situation also happens when considering the eigenvalue problem (\ref{a10}) in the hyperbolic space $\mathbb{H}^n$ -- see Theorem \ref{theo-2}. \\
(2) It is well-known that the Steklov eigenvalue problem (\ref{a11})
has discrete spectrum and all the eigenvalues (in this discrete
spectrum) can be listed non-decreasingly as follows:
\begin{eqnarray*} \label{sequence-2}
0=\sigma_{0}(\Omega)<\sigma_{1}(\Omega)\leq\sigma_{2}(\Omega)\leq\sigma_{3}(\Omega)\leq\cdots\uparrow\infty.
 \end{eqnarray*}
Clearly, if $\phi=const.$, then
$\sigma_{i,\phi}(\Omega)=\sigma_{i}(\Omega)$ for each
$i=0,1,2,\cdots$. Then by Theorem \ref{theo-1} and Corollary
\ref{coro-1}, for a ball $B_{\Omega}\subset\mathbb{R}^n$
with\footnote{~In the sequel, by the abuse of notations and without
confusion, $|\cdot|$ would stand for the Hausdorff measure of a
given geometric object.} $|B_{\Omega}|=|\Omega|$ (i.e. $B_{\Omega}$
has the same volume as $\Omega$), one directly has
\begin{eqnarray} \label{II-1-add}
\frac{1}{\sigma_{1}(\Omega)}+\frac{1}{\sigma_{2}(\Omega)}+\cdots+\frac{1}{\sigma_{n-1}(\Omega)}\geq\frac{n-1}{\sigma_{1}(B_{\Omega})}
\end{eqnarray}
 and
\begin{eqnarray} \label{II-2-add}
 \sigma_{1}(\Omega)\leq\sigma_{1}(B_{\Omega}),
\end{eqnarray}
 where the equality in (\ref{II-1-add}) or (\ref{II-2-add}) holds if and only if $\Omega$ is
 the ball $B_{\Omega}$. Moreover, the estimate
 (\ref{II-2-add}) tells us that:
  \begin{itemize}
 \item \emph{Among all bounded domains in
$\mathbb{R}^n$ having the same volume, the ball maximizes the first
nonzero Steklov eigenvalue of the Laplacian.}
  \end{itemize}
 This is exactly the reason why we usually call (\ref{II-1-add})  and
 (\ref{II-2-add}) the \emph{spectral isoperimetric inequalities}. Of
 course, in (\ref{II-1-add}) or (\ref{II-2-add}), there is no need
 to give the information of the center for $B_{\Omega}$, since for
 balls having the same radius in $\mathbb{R}^n$, they  have exactly the
 same first nonzero Steklov eigenvalue $\sigma_{1}$ of the Laplacian (corresponding to
 $\phi=const.$). \\
(3) Now, in order to show the meaning and importance of our estimate
(\ref{II-1}), we would like to recall some interesting results.  For
a simply connected domain $\Omega\subset\mathbb{R}^2$, by using
conformal mapping, Weinstock \cite{RW} in 1954 showed that
 \begin{eqnarray} \label{II-ex1}
 \sigma_{1}(\Omega)\leq\frac{2\pi}{|\partial\Omega|}=\frac{|\partial\mathbb{B}^2|}{|\partial\Omega|}\sigma_{1}(\mathbb{B}^2),
 \end{eqnarray}
where $\mathbb{B}^2$ denotes an arbitrary disk in the plane
$\mathbb{R}^2$. It is not hard to see that the quantity
$|\partial\mathbb{B}^2|\sigma_{1}(\mathbb{B}^2)$ is scale invariant.
The equality in (\ref{II-ex1}) holds if and only if $\Omega$ is a
round ball. By using the classical geometric isoperimetric
inequalities in the Euclidean space (see, e.g., \cite[Chapter
1]{RS}), one can obtain directly from the estimate (\ref{II-ex1})
that the spectral isoperimetric inequality (\ref{II-2-add}) holds
for simply connected planar domains. 14 years later, Hersch and
Payne \cite{HP} found that from Weinstock's proof, one can get a
more sharper estimate
 \begin{eqnarray}  \label{II-ex2}
\frac{1}{\sigma_{1}(\Omega)}+\frac{1}{\sigma_{2}(\Omega)}\geq\frac{2|\partial\mathbb{B}^2|}{|\partial\Omega|}\sigma_{1}(\mathbb{B}^2)
 \end{eqnarray}
for simply connected domain $\Omega\subset\mathbb{R}^2$. By using
the classical geometric isoperimetric inequalities in
$\mathbb{R}^2$, it follows from (\ref{II-ex2}) that
 \begin{eqnarray}  \label{II-ex3}
\frac{1}{\sigma_{1}(\Omega)}+\frac{1}{\sigma_{2}(\Omega)}\geq\frac{2}{\sigma_{1}(B_{\Omega})}
\end{eqnarray}
 holds for any simply connected domain $\Omega\subset\mathbb{R}^2$,
 where $B_{\Omega}\subset\mathbb{R}^2$ is a disk satisfying
 $|B_{\Omega}|=|\Omega|$. The isoperimetric inequality
 (\ref{II-ex3}) has been generalized to higher dimensional case in
 2001 by Brock \cite{FB}. In fact, for a bounded domain $\Omega$ (with
 Lipschitz boundary $\partial\Omega$) in $\mathbb{R}^n$, Brock
  considered the eigenvalue problem (\ref{a11}) with the
 boundary condition replaced by
  \begin{eqnarray}  \label{BC-1}
 \frac{\partial u}{\partial \vec{\eta}}=\sigma\rho
u~~\mathrm{on}~~\partial \Omega,
  \end{eqnarray}
where $\rho\in L^{\infty}(\partial\Omega)$ is a given nonnegative
 function on $\partial\Omega$, and he \cite[Theorem 1]{FB}
proved that
 \begin{eqnarray}  \label{II-ex4}
\frac{1}{\sigma_{1}(\Omega)}+\frac{1}{\sigma_{2}(\Omega)}+\cdots+\frac{1}{\sigma_{n-1}(\Omega)}+\frac{1}{\sigma_{n}(\Omega)}\geq
n\bar{\rho}R_{B_{\Omega}},
 \end{eqnarray}
 where $R_{B_{\Omega}}$ is the radius of a ball
 $B_{\Omega}\subset\mathbb{R}^{n}$ satisfying the volume constraint
 $|B_{\Omega}|=|\Omega|$, and $\bar{\rho}$ is given by
  \begin{eqnarray}  \label{r-def}
\bar{\rho}=\frac{nw_{n}(R_{B_{\Omega}})^{n-1}}{\int_{\partial\Omega}\rho^{-1}\widehat{dv}},
  \end{eqnarray}
with $w_{n}$ the volume of the unit ball in $\mathbb{R}^n$. The
equality in (\ref{II-ex4}) is attained if $\Omega$ is a ball and if
$\rho$ is constant on $\partial\Omega$. Clearly, if $\rho\equiv1$,
then from (\ref{r-def}) one has in this setting that
$\bar{\rho}=nw_{n}(R_{B_{\Omega}})^{n-1}/|\partial\Omega|$, and then
(\ref{II-ex4}) becomes
 \begin{eqnarray}   \label{II-ex5}
\frac{1}{\sigma_{1}(\Omega)}+\frac{1}{\sigma_{2}(\Omega)}+\cdots+\frac{1}{\sigma_{n-1}(\Omega)}+\frac{1}{\sigma_{n}(\Omega)}\geq
nR_{B_{\Omega}}\frac{|\partial
B_{\Omega}|}{|\partial\Omega|}=\frac{n}{\sigma_{1}(B_{\Omega})}\frac{|\partial
B_{\Omega}|}{|\partial\Omega|}.
 \end{eqnarray}
Moreover, when $\rho\equiv1$, the boundary condition (\ref{BC-1})
degenerates into the classical one used in the Steklov eigenvalue
problem (\ref{a11}). In fact, from mathematicians' viewpoint, there
might be no big difference between  the boundary condition
(\ref{BC-1}) and the one used in (\ref{a11}). However, the quantity
$\int_{\partial\Omega}\rho\widehat{dv}$ is called the \emph{total
mass} and it do has physical meaning when the system
\begin{eqnarray} \label{a12}
\left\{\begin{array}{ll} \Delta u
=0&~~\mathrm{in} ~~\Omega\subset \mathbb{R}^{n}, \\[2mm]
\frac{\partial u}{\partial \vec{\eta}}=\sigma\rho
u&~~\mathrm{on}~~\partial \Omega
\end{array}\right.
\end{eqnarray}
is used to describe some physical phenomenon. This is the reason why
some mathematicians insist on using the boundary condition
(\ref{BC-1}) rather than the classical one for the Steklov
eigenvalue problem of the Laplacian. As said by Brock \cite[page
70]{FB} that \emph{if $\rho$ is constant, then a slight modification
of the proof of } the estimate (\ref{II-ex4}) would lead to the
result that
 \begin{eqnarray}   \label{II-ex6}
\frac{1}{\sigma_{1}(\Omega)}+\frac{1}{\sigma_{2}(\Omega)}+\cdots+\frac{1}{\sigma_{n-1}(\Omega)}+\frac{1}{\sigma_{n}(\Omega)}\geq
\frac{n}{\sigma_{1}(B_{\Omega})}
 \end{eqnarray}
holds for\footnote{~For the boundary condition (\ref{BC-1}), it is
easy to see that there is no essential difference between
$\rho=const.$ and $\rho\equiv1$.} $\rho\equiv1$. It is not hard to
see that the estimate (\ref{II-ex6}) is sharper than the estimate
(\ref{II-ex4}) for the case $\rho\equiv1$. This is because under the
fixed volume constraint $|B_{\Omega}|=|\Omega|$, by the classical
geometric isoperimetric inequality in $\mathbb{R}^n$, one has
$|\partial B_{\Omega}|\leq|\partial\Omega|$, with equality if and
only if $\Omega$ is a ball congruent to $B_{\Omega}$. Based on the
above two reasons, in this paper we just look at the classical
boundary condition $\frac{\partial u}{\partial \vec{\eta}}=\sigma u$
for the Steklov-type eigenvalue problem (\ref{a10}) such that our
spectral isoperimetric inequalities (\ref{II-1}) and (\ref{II-1-1})
appear as their present versions which directly imply the nice
spectral isoperimetric inequalities (\ref{II-2}) and (\ref{II-2-2}).
Actually, one can also consider a slightly general modification to
the problem (\ref{a10}) as follows
\begin{eqnarray} \label{a10-add}
\left\{\begin{array}{ll} \L u
=0&~~\mathrm{in} ~~\Omega\subset M^{n}, \\[2mm]
\frac{\partial u}{\partial \vec{\eta}}=\sigma\rho
u&~~\mathrm{on}~~\partial \Omega,
\end{array}\right.
\end{eqnarray}
with $\rho\in L^{\infty}(\partial\Omega)$ a given nonnegative
function on $\partial\Omega$, and similar to the situation happened
in \cite{FB}, under suitable assumptions on the weighted function
$\phi$ and the given function $\rho$, (by making a slight
modification to the argument we have shown in Section \ref{S3}) a
lower bound involving an integration of $\rho$, $\phi$ and
$\sigma_{1}(B_{\Omega})$ would be obtained for
$\sum_{i=1}^{n-1}\frac{1}{\sigma_{i,\phi}(\Omega)}$ when the given
Riemnnain $n$-manifold $M^{n}$ is chosen to be $\mathbb{R}^n$ or a
hyperbolic $n$-space. The second named author, my student S. J.
Zhang, will show the details of this fact in his dissertation
\cite{SJZ}. \\
(4) Obviously, by using the monotonicity of Steklov eigenvalues of
the Laplacian, for any bounded domain $\Omega\subset\mathbb{R}^{n}$
with Lipschitz boundary, (under the volume constraint
$|\Omega|=|B_{\Omega}|$) the spectral isoperimetric inequality
(\ref{II-2-add}) can be also obtained from the estimate
(\ref{II-ex6}). However, since on \cite[page 70]{FB} there is no
rigidity description for the equality case in (\ref{II-ex6}), one
cannot get the rigidity result for the estimate (\ref{II-2-add}) by
directly using (\ref{II-ex6}). \\
(5) A natural question is:
 \begin{itemize}
 \item \emph{Whether or not the spectral isoperimetric
inequality (\ref{II-2-add}) could still hold in the situation that
$M^{n}$ is a space form or a general Riemannian $n$-manifold?}
 \end{itemize}
 There were a few of results have been obtained. For instance, when
 $M^{n}$ was chosen to be a $2$-dimensional space form
 $\mathbb{M}^{2}(\kappa)$ (with constant sectional curvature
 $\kappa$), Escobar \cite[Theorem 7]{JFE} showed the validity of the spectral isoperimetric
inequality (\ref{II-2-add}); when $M^{n}$ was chosen to be a
non-compact rank one symmetric space such that the sectional
curvature $\mathrm{Sec}(M^{n})$ satisfies
$\mathrm{Sec}(M^{n})\in[-4,-1]$, the validity of (\ref{II-2-add})
has also been obtained in \cite[Theorem 1.1]{RS}.  Quantitative
versions of (\ref{II-2-add}) were also obtained by Brasco-De
Philippis-Ruffini \cite{BPR} for the Euclidean space and by
Castillon-Ruffini \cite{CR} for the non-compact rank one symmetric
space. Recently, Chen \cite{HC} has made an interesting progress on
this question, that is,
\begin{itemize}
\item Let $M^{n}$ be an $n$-dimensional complete simply connected Riemannian
manifold whose sectional curvature and Ricci curvature separately
satisfy $\mathrm{Sec}(M^{n})\leq\kappa$ and $\mathrm{Ric}(M^{n})\geq
K$, with $K\leq\kappa\leq0$, and let $\Omega\subset M^{n}$ be a
bounded domain with smooth boundary $\partial\Omega$. Then for the
Steklov eigenvalue problem (\ref{a11}), he \cite[Theorem 1.2 and
Corollary 1.3]{HC} showed that the inequality
\begin{eqnarray} \label{II-ex7}
\frac{1}{\sigma_{1}(\Omega)}+\frac{1}{\sigma_{2}(\Omega)}+\cdots+\frac{1}{\sigma_{n-1}(\Omega)}\geq\frac{1}{C^2}\frac{n-1}{\sigma_{1}(B_{\Omega})}
\end{eqnarray}
holds, where $B_{\Omega}$ is the geodesic ball in the space form
$\mathbb{M}^{n}(\kappa)$ satisfying the volume constraint
$|B_{\Omega}|=|\Omega|$, $C:=C(n,K,\kappa,\mathrm{diam}(\Omega))$ is
a positive constant defined by (1.12) of \cite{HC} and depending on
$n$, $K$, $\kappa$, the diameter $\mathrm{diam}(\Omega)$ of the
domain $\Omega$. Specially, as pointed out in \cite[Corollary
1.4]{HC}, when $M^{n}$ was chosen to be $\mathbb{R}^n$ or
$\mathbb{H}^{n}$, $C\equiv1$ and then the estimate (\ref{II-ex7})
becomes
\begin{eqnarray} \label{II-ex8}
\frac{1}{\sigma_{1}(\Omega)}+\frac{1}{\sigma_{2}(\Omega)}+\cdots+\frac{1}{\sigma_{n-1}(\Omega)}\geq\frac{n-1}{\sigma_{1}(B_{\Omega})}.
\end{eqnarray}
Moreover, the equality in (\ref{II-ex8}) if and only if $\Omega$ is
isometric to $B_{\Omega}$.
 \end{itemize}
It is not hard to see that our Brock-type spectral isoperimetric
inequalities (\ref{II-1}), (\ref{II-1-1}) in Theorems \ref{theo-1}
 and \ref{theo-2} can be seen as a direct weighted version of
 (\ref{II-ex8}). One might ask whether or not one could improve our
 main results of this paper in $\mathbb{R}^n$ and $\mathbb{H}^n$ to
 curved spaces with the curvature pinching assumption imposed as in
 \cite{HC}. The answer should be affirmative and we prefer to leave
 this attempt to readers who have interest in this topic.
 The reason that we think our
 main results of this paper in $\mathbb{R}^n$ and $\mathbb{H}^n$ can
 be improved to curved spaces is the following: in \cite{CM1}, Chen
 and Mao have proven an isoperimetric inequality for lower order nonzero Neumann
eigenvalues of the Witten-Laplacian on bounded domains in a
Euclidean space or a hyperbolic space; by making extra suitable
assumption on the symmetrized correspondence of the weighted
function $\phi$ (under the constraint of fixed weighted volume)
through the Schwarz symmetrization, the main conclusions in
\cite{CM1} can be extended to curved spaces $M^{n}$ with curvature
pinching $\mathrm{Sec}(M^{n})\leq\kappa$ and
$\mathrm{Ric}(M^{n})\geq K$, $K\leq\kappa\leq0$ (see \cite{DCM}).
Based on the relation between the Neumann eigenvalues and the
Steklov eigenvalues of the Witten-Laplacian, we affirm the
possibility of improving our
 main results of this paper in $\mathbb{R}^n$ and $\mathbb{H}^n$ to
 curved spaces with suitable curvature pinching.
 \\
(6) Based on the deriving process of our main conclusions in
Theorems \ref{theo-1} and \ref{theo-2}, together with the existence
of Brock's isoperimetric inequality (\ref{II-ex6}) in
$\mathbb{R}^n$, we would like to propose the following open problem,
which we think it should be suitable to call it \emph{\textbf{the
Brock-type conjecture} about lower Steklov-type eigenvalues of the
Witten-Laplacian}.
\begin{itemize}
\item  \textbf{Question A}. Consider the eigenvalue problem
(\ref{a10}) with choosing $M^n$ to be $M^{n}=\mathbb{R}^n$ (or
$M^{n}=\mathbb{H}^n$), and assume that the function $\phi$ satisfies
\textbf{Property I}. Let $\Omega$ be a bounded domain with Lipschitz
continuous boundary in
 $\mathbb{R}^n$ (or $\mathbb{H}^n$), and let $B_{R}(o)$ be a (geodesic) ball of radius $R$ and centered at the
 origin $o$
 of
 $\mathbb{R}^{n}$ (or $\mathbb{H}^n$) such that $|\Omega|_{\phi}=|B_{R}(o)|_{\phi}$,
 i.e. $\int_{\Omega}d\eta=\int_{B_{R}(o)}d\eta$.  Then
  \begin{eqnarray*}
\frac{1}{\sigma_{1,\phi}(\Omega)}+\frac{1}{\sigma_{2,\phi}(\Omega)}+\cdots+\frac{1}{\sigma_{n-1,\phi}(\Omega)}+\frac{1}{\sigma_{n,\phi}(\Omega)}\geq\frac{n}{\sigma_{1,\phi}(B_{R}(o))}.
\end{eqnarray*}
The equality case holds if and only if $\Omega$ is the ball
$B_{R}(o)$ (or $\Omega$ is isometric to the geodesic ball
$B_{R}(o)$).
\end{itemize}
 }
\end{remark}

The paper is organized as follows. In Section \ref{S2}, we give
several properties related to the first nonzero Steklov eigenvalues
and their eigenfunctions of the Witten-Laplacian on geodesic balls
in the space forms.  A proof of Theorems \ref{theo-1} and
\ref{theo-2} will be given in  Section \ref{S3}.

\section{Several useful facts}  \label{S2}
\renewcommand{\thesection}{\arabic{section}}
\renewcommand{\theequation}{\thesection.\arabic{equation}}
\setcounter{equation}{0}

In this section, we need several facts which make an important role
in the derivation of main conclusions in  Theorems \ref{theo-1} and
\ref{theo-2}. First, we need:

\begin{theorem} \label{theorem2-1}
 Let
$\mathbb{M}^{n}(\kappa)$ be an $n$-dimensional complete simply
connected Riemannian manifold of constant sectional curvature
$\kappa\in\{1,0,-1\}$, and let $B_{R}(o)$ be a geodesic ball
centered at the base point $o\in\mathbb{M}^{n}(\kappa)$ and with
radius $R$. Assume that $\phi$ is a radial function w.r.t. the
distance parameter $t:=d(o,\cdot)$. Then the first nonzero Steklov
eigenvalue $\sigma_{1,\phi}(B_{R}(o))$  of the Witten-Laplacian on
$B_{R}(o)$ is given by
 \begin{eqnarray*}
\sigma_{1,\phi}(B_{R}(o))=\frac{\int_{B_{R}(o)}\left[T^{2}(t)v_{1}(\partial
B_{R}(o))+(T'(t))^{2}\right]d\mu}{T^{2}(R)|\partial
B_{R}(o)|_{\phi}},
\end{eqnarray*}
 and the functions $T(t)\frac{x_{j}}{t}$, $j=1,2,\cdots,n$, are eigenfunctions corresponding to the eigenvalue
 $\sigma_{1,\phi}(B_{R}(o))$, where $x_{j}$,
$j=1,2,\cdots,n$, are coordinate functions of the globally defined
orthonormal coordinate system set up in $\mathbb{M}^{n}(\kappa)$,
$v_{1}(\partial B_{R}(o))$ is the first nonzero closed eigenvalue of
the Laplacian on the sphere $\partial B_{R}(o)$, $|\partial
B_{R}(o)|_{\phi}$ denotes the weighted measure of $\partial
B_{R}(o)$, and $T(t)$ satisfies
\begin{eqnarray}  \label{2-1}
\left\{
\begin{array}{ll}
T''+\left(\frac{(n-1)C_{\kappa}}{S_{\kappa}}-\phi'\right)T'-(n-1)S_{\kappa}^{-2}T=0,\\[1mm]
T(0)=0,~T'(0)=1.
\end{array}
\right.
\end{eqnarray}
 Here $C_{\kappa}(t)=\left(S_{\kappa}(t)\right)'$ and
 \begin{eqnarray*}
S_{\kappa}(t)=\left\{
\begin{array}{lll}
\sin t,~~&\mathrm{if}~\mathbb{M}^{n}(\kappa)=\mathbb{S}^{n}_{+}, \\
t,~~&\mathrm{if}~\mathbb{M}^{n}(\kappa)=\mathbb{R}^{n},\\
\sinh t,~~&\mathrm{if}~\mathbb{M}^{n}(\kappa)=\mathbb{H}^{n},
\end{array}
\right.
 \end{eqnarray*}
  with $\mathbb{S}^{n}_{+}$ the $n$-dimensional hemisphere of radius $1$.
\end{theorem}

\begin{remark}
\rm{ In \cite[Remark 1.4]{CM}, the concept ``\emph{base point}" has
been introduced clearly for spherically symmetric manifolds, which
were called \emph{generalized space forms} by Katz-Kondo \cite{KK},
and readers who are interested in this concept can check there for
details. As also explained clearly in \cite[Remarks 1.4 and 1.6]{CM}
that: (1) when $\mathbb{M}^{n}(\kappa)$ is chosen to be
$\mathbb{R}^n$ or $\mathbb{H}^n$, since in this case the space is
two-point homogenous, any point in the space can be seen as the base
point which leads to the situation that for convenience, we need to
additionally require that $o$ is the origin of the globally defined
coordinate system; (2) when $\mathbb{M}^{n}(\kappa)$ is chosen to be
$\mathbb{S}^{n}_{+}$, the base point $o$ is uniquely determined such
that $\mathbb{S}^{n}_{+}$ is spherically symmetric w.r.t. the base
point. }
\end{remark}

\begin{proof}
Let $\{G_{i}\}_{i=0}^{\infty}$ be a complete orthogonal
 set consisting of
eigenfunctions of the Laplacian on the sphere $\partial B_{R}(o)$
with associated closed eigenvalues $v_{i}(\partial
B_{R}(o))=i(i+n-2)/R^2$, $i=0,1,2,\cdots$. Let
$T:=T_{1},T_{2},\cdots$ be radial functions corresponding to
$v_{i}(\partial B_{R}(o))$ and satisfying the following system
 \begin{eqnarray}  \label{ODE-21}
\left\{
\begin{array}{ll}
\frac{(n-1)C_{\kappa}(t)}{S_{\kappa}(t)}T'_{i}(t)+T''_{i}(t)-\frac{v_{i}}{S_{\kappa}^{2}(t)}T_{i}(t)-\phi'(t)
T'_{i}(t)=0, ~~~t\in(0,R) \\ [2mm]
 T_{i}(0)=0, ~~T'_{i}(R)=\beta_{i}T_{i}(R),
\end{array}
\right.
 \end{eqnarray}
where $\beta_{i}$ are eigenvalues determined later. Let
$u_{0,\phi}=1$ and $u_{i,\phi}(t,\xi)=T_{i}(t)G_{i}(\xi)$,
$\xi\in\mathbb{S}^{n-1}$ with $\mathbb{S}^{n-1}$ the
$(n-1)$-dimensional unit Euclidean
 sphere, $i=1,2,\cdots$. Then the set
$\{u_{i,\phi}\}_{i=0}^{\infty}$ forms an orthogonal basis of
$\widetilde{L}^2(B_{R}(o))$ (i.e., the set of all $L^2$ integrable
functions defined over $B_{R}(o)$ w.r.t. the weighted density
$d\mu$).

Noticing from (\ref{ODE-21}) that $\L u_{i,\phi}=0$ on $B_{R}(o)$
for all $i\geq 1$, we have by applying the divergence theorem that
 \begin{eqnarray} \label{2-3}
\int_{\partial B_{R}(o)}u_{i,\phi}\widehat{d\mu}&=&\int_{\partial B_{R}(o)}T_{i}G_{i}\widehat{d\mu}\nonumber\\
&=&T_{i}(R)e^{-\phi(R)}\int_{\partial B_{R}(o)}G_{i}\widehat{dv}\nonumber\\
&=&-\frac{T_{i}(R)}{v_{i}}e^{-\phi(R)}\int_{\partial B_{R}(o)}\Delta_{\partial B_{R}(o)}G_{i}\widehat{dv}\nonumber\\
&=&0,
 \end{eqnarray}
where $\Delta_{\partial B_{R}(o)}$ denotes the Laplacian on the
geodesic sphere $\partial B_{R}(o)$.

On the other hand, by a direct calculation, and together with the
divergence theorem, one has
 \begin{eqnarray}  \label{2-4}
&&\int_{B_{R}(o)}|\nabla u_{i,\phi}|^{2}d\mu \nonumber\\
&=&\int_{B_{R}(o)}\mathrm{div}(u_{i,\phi}\nabla u_{i,\phi}e^{-\phi})dv-\int_{B_{R}(o)}u_{i,\phi}\Delta u_{i,\phi}d\mu-\int_{B_{R}(o)}u_{i,\phi}\nabla u_{i,\phi}\nabla e^{-\phi}dv\nonumber\\
&=&\int_{\partial B_{R}(o)}u_{i,\phi}\frac{\partial u_{i,\phi}}{\partial\vec{\eta}}\widehat{d\mu}-\int_{B_{R}(o)}u_{i,\phi}\Delta u_{i,\phi}d\mu+\int_{B_{R}(o)}u_{i,\phi}\nabla u_{i,\phi}\nabla\phi d\mu\nonumber\\
&=&\int_{\partial B_{R}(o)}u_{i,\phi}\frac{\partial u_{i,\phi}}{\partial\vec{\eta}}\widehat{d\mu}-\int_{B_{R}(o)}u_{i,\phi}\L u_{i,\phi}d\mu\nonumber\\
&=&\int_{\partial B_{R}(o)}T_{i}G_{i}\frac{\partial T_{i}G_{i}}{\partial t}\widehat{d\mu}\nonumber\\
&=&e^{-\phi(R)}T_{i}(R)\int_{\partial B_{R}(o)}G_{i}\langle\nabla T_{i}G_{i},\partial_{t}\rangle \widehat{dv}\nonumber\\
&=&e^{-\phi(R)}T_{i}(R)\int_{\partial B_{R}(o)}\left(G_{i}^{2}\langle\nabla T_{i},\partial_{t}\rangle+T_{i}G_{i}\langle\nabla G_{i},\partial_{t}\rangle\right)\widehat{dv}\nonumber\\
&=&e^{-\phi(R)}T_{i}(R)\int_{\partial B_{R}(o)}G_{i}^{2}\langle\nabla T_{i},\partial_{t}\rangle \widehat{dv}\nonumber\\
&=&e^{-\phi(R)}T_{i}(R)T'_{i}(R)\int_{\partial
B_{R}(o)}G_{i}^{2}\widehat{dv},
 \end{eqnarray}
 where $\mathrm{div}$ is the divergence operator, and $\partial_{t}$ denotes the radial tangent vector induced by
 the distance parameter $t=d(\cdot,o)$. Therefore, by the
 variational
 characterizations (\ref{cha-1})-(\ref{cha-2}), and together with (\ref{2-3})-(\ref{2-4}), we can obtain
 \begin{eqnarray}  \label{2-5}
\frac{\int_{B_{R}(o)}|\nabla u_{i,\phi}|^{2}d\mu}{\int_{\partial
B_{R}(o)}u_{i,\phi}^{2}\widehat{d\mu}}=\frac{e^{-\phi(R)}T_{i}(R)T'_{i}(R)\int_{\partial
B_{R}(o)}G_{i}^{2}\widehat{dv}}{T_{i}^{2}(R)e^{-\phi(R)}\int_{\partial
B_{R}(o)}G_{i}^{2}\widehat{dv}}
=\frac{T'_{i}(R)}{T_{i}(R)}=\beta_{i}.
 \end{eqnarray}
 However, by a direct calculation, we also have
  \begin{eqnarray}  \label{2-6}
|\nabla u_{i,\phi}|^{2}&=&|\nabla^{\mathbb{S}^{n-1}(t)}u_{i,\phi}|^{2}+\langle\nabla u_{i,\phi},\partial_{t}\rangle^{2}\nonumber\\
&=&|\nabla^{\mathbb{S}^{n-1}(t)}(T_{i}G_{i})|^{2}+G_{i}^{2}(T'_{i})^{2}\nonumber\\
&=&T_{i}^{2}|\nabla^{\mathbb{S}^{n-1}(t)}G_{i}|^{2}+G_{i}^{2}(T'_{i})^{2}\nonumber\\
&=&T_{i}^{2}\left(\frac{1}{2}\L G_{i}^{2}-G_{i}\L G_{i}\right)+G_{i}^{2}(T'_{i})^{2}\nonumber\\
&=&T_{i}^{2}\left(\frac{1}{2}\Delta_{\phi}G_{i}^{2}+v_{i}(\partial
B_{R}(o))G_{i}^{2}\right)+G_{i}^{2}(T'_{i})^{2},
  \end{eqnarray}
where $\nabla^{\mathbb{S}^{n-1}(t)}$ denotes the gradient operator
on $\mathbb{S}^{n-1}(t)$, with $\mathbb{S}^{n-1}(t)$ the Euclidean
$(n-1)$-sphere of radius $t$. Since
$$\frac{1}{2}\int_{B_{R}(o)}\L(G_{i}^{2})dv=\int_{\partial
B_{R}(o)}G_{i}\langle\nabla
G_{i},\partial_{t}\rangle\widehat{dv}=0,$$ one can get from
(\ref{2-6}) that
 \begin{eqnarray}  \label{2-7}
\beta_{i}&=&\frac{\int_{B_{R}(o)}|\nabla
u_{i,\phi}|^{2}d\mu}{\int_{\partial
B_{R}(o)}u_{i,\phi}^{2}\widehat{d\mu}}\nonumber\\
&=&\frac{\int_{B_{R}(o)}\left[T_{i}^{2}\left(\frac{1}{2}\L G_{i}^{2}+v_{i}(\partial B_{R}(o))G_{i}^{2}\right)+G_{i}^{2}(T'_{i})^{2}\right]d\mu}{\int_{\partial B_{R}(o)}u_{i,\phi}^{2}\widehat{d\mu}}\nonumber\\
&=&\frac{\int_{B_{R}(o)}\left(T_{i}^{2}v_{i}G_{i}^{2}+G_{i}^{2}(T'_{i})^{2}\right)d\mu}{\int_{\partial
B_{R}(o)}u_{i,\phi}^{2}\widehat{d\mu}}.
 \end{eqnarray}
Since $v_{i}(\partial B_{R}(o))\geq v_{1}(\partial B_{R}(o))$ for
$i\geq 1$, by applying the Strum-Liouville comparison theorem  to
(\ref{2-7}) we have $\beta_{i}\geq\beta_{1}$.  Meanwhile, all the
admissible functions in the variational characterization of
$\sigma_{1}(B_{R}(o))$ are orthogonal to nonzero constant functions
on $\partial B_{R}(o)$. So, $\sigma_{1}(B_{R}(o))=\beta_{1}$.
Consider the geodesic normal coordinate system $\{x_{i}\}_{i=1}^{n}$
 centered at $o$, and then $\frac{x_{j}}{t}$, $j=1,2,\cdots,n$, are eigenfunctions
 corresponding to the first nonzero eigenvalue $v_{1}(\partial B_{R}(o))$ of the geodesic sphere $\partial
 B_{R}(o)$. Therefore, from (\ref{2-7}) we can obtain
 \begin{eqnarray*}
\sigma_{1,\phi}(B_{R}(o))\int_{\partial
B_{R}(o)}T^{2}\left(\frac{x_{j}}{t}\right)^{2}\widehat{d\mu}=\int_{B_{R}(o)}\left[T^{2}v_{1}(\partial
B_{R}(o))\left(\frac{x_{j}}{t}\right)^{2}+\left(\frac{x_{j}}{t}\right)^{2}(T')^{2}\right]d\mu.
 \end{eqnarray*}
Summing both sides of the above equality over the index $j$ from $1$
to $n$ yields
 \begin{eqnarray*}
\sigma_{1,\phi}(B_{R}(o))&=&\sum_{j=1}^{n}\frac{\int_{B_{R}(o)}\left[T^{2}(t)v
_{1}(\partial
B_{R}(o))(\frac{x_{j}}{t})^{2}+(\frac{x_{j}}{t})^{2}(T')^{2}\right]d\mu}{\int_{\partial
B_{R}(o)}T^{2}(t)\left(\frac{x_{j}}{t}\right)^{2}\widehat{d\mu}}\\
&=&\frac{\int_{B_{R}(o)}\left[T^{2}(t)v_{1}(\partial
B_{R}(o))+(T')^{2}(t)\right)d\mu}{\int_{\partial
B_{R}(o)}T^{2}(t)\widehat{d\mu}}\\
&=&\frac{\int_{B_{R}(o)}\left[T^{2}(t)v_{1}(\partial
B_{R}(o))+(T')^{2}(t)\right)d\mu}{T^{2}(R)|\partial
B_{R}(o)|_{\phi}}
 \end{eqnarray*}
which implies the first assertion of Theorem \ref{theorem2-1}. The
derivation of the system (\ref{2-1}) follows from (\ref{ODE-21}),
$T=T_{1}$, (\ref{2-5}) and $\sigma_{1}(B_{R}(o))=\beta_{1}$
directly.
\end{proof}

We also need the following monotone property.

\begin{lemma}  \label{lemma2-3}
Let $T(t)$ be the solution to the system (\ref{2-1}). Assume that
$\kappa=0$ or $\kappa=-1$ (i.e. $\mathbb{M}^{n}(\kappa)$ was chosen
to be $\mathbb{R}^{n}$ or $\mathbb{H}^{n}$) and the weighted
function $\phi(t)$ satisfies $\phi'(t)\leq 0$ and $\phi''(t)\leq 0$.
Define two radial functions $G$ and $H$ as follows
\begin{eqnarray}
&&G(t):=\left(T^{2}(t)\right)'+(n-1)\frac{C_{\kappa}(t)}{S_{\kappa}(t)}T^{2}(t)-T^{2}(t)\phi'(t), \label{2-8}\\
&&H(t):=(T'(t))^{2}+(n-1)\frac{T^{2}(t)}{S_{\kappa}^{2}(t)}.
\label{2-9}
\end{eqnarray}
Then $G$ is nonnegative and non-decreasing on $[0,+\infty)$, and $H$
is nonnegative and non-increasing on $[0,+\infty)$.
\end{lemma}

\begin{proof}
By the maximum principle, we obtain $T(t)>0$ on $(0,R)$. We also
have $T'(t)>0$ on $(0,R)$. Indeed, if $T'(t)$ is not always positive
 on $(0,R)$, then $T(t)$ is not always increasing on
$(0,R)$. In this situation, there must exist a point
$t_{0}\in(0,R)$, such that $T'(t_{0})=0$ and $T''(t_{0})\leq 0$.
Applying the system (\ref{2-1}) at $t_{0}$, we can obtain
\begin{eqnarray*}
T''(t_{0})+\left((n-1)\frac{C_{\kappa}(t_{0})}{S_{\kappa}(t_{0})}-\phi'(t_{0})\right)T'(t_{0})-\frac{n-1}{S_{\kappa}^{2}(t_{0})}T(t_{0})=0,
\end{eqnarray*}
and from which
$T''(r_{0})=\frac{n-1}{S_{\kappa}^{2}(r_{0})}T(r_{0})>0$ holds.
However, this is a contradiction. So,  $T'(t)>0$ on $(0,R)$. Since
$T(t)$ is non-negative and increasing on $[0,+\infty)$,
$\phi'(t)\leq 0$ on $[0,+\infty)$, we have that functions $G$ and
$H$ are non-negative on $[0,+\infty)$. Using the system (\ref{2-1})
and noticing that $\phi''(t)\leq 0$, on $(0,+\infty)$ we have
\begin{eqnarray}  \label{2-10}
G'(t)&=&2T(t)T''(t)+2(T'(t))^{2}+(n-1)\left(2\frac{C_{\kappa}(t)}{S_{\kappa}(t)}T(t)T'(t)+\frac{S''_{\kappa}(t)}{S_{\kappa}(t)}T^{2}(t)-
\frac{(S'_{\kappa}(t))^{2}}{S_{\kappa}^{2}(t)}T^{2}(t)\right)   \nonumber\\
&& -2T(t)T'(t)\phi'(t)-T^{2}(t)\phi''(t)\nonumber\\
&=&2T(t)\left[-\left((n-1)\frac{C_{\kappa}(t)}{S_{\kappa}(t)}-\phi'(t)\right)T'(t)+\frac{n-1}{S_{\kappa}^{2}(t)}T(t)\right]+2(T'(t))^{2}+ \nonumber\\
&&(n-1)\left(2\frac{C_{\kappa}(t)}{S_{\kappa}(t)}T(t)T'(t)+\frac{S''_{\kappa}(t)}{S_{\kappa}(t)}T^{2}(t)-\frac{(C_{\kappa}(t))^{2}}{S_{\kappa}^{2}(t)}T^{2}(t)\right)- \nonumber\\
&&2T(t)T'(t)\phi'(t)-T^{2}(t)\phi''(t) \nonumber\\
&=& 2(T'(t))^{2}+\frac{(n-1)T^{2}(t)}{S_{\kappa}^{2}(t)}\left[2+S_{\kappa}(t)S''_{\kappa}(t)-(C_{\kappa}(t))^{2}\right]-T^{2}(t)\phi''(t) \nonumber\\
&=& 2(T'(t))^{2}+\frac{(n-1)T^{2}(t)}{S_{\kappa}^{2}(t)}-T^{2}(t)\phi''(t) \nonumber\\
&\geq&0,
\end{eqnarray}
where in the last equality of (\ref{2-10}) we have used the identity
$S_{\kappa}(t)S''_{\kappa}(t)-(C_{\kappa}(t))^{2}=-1$ for all
$\kappa\in\mathbb{R}$.  Hence, $G(t)$ is non-decreasing on
$(0,+\infty)$. Likewise, noticing that $\phi'(t)\leq 0$ and
$S'_{\kappa}(t)=C_{\kappa}(t)\geq 1$ for $\mathbb{R}^{n}$ or
$\mathbb{H}^{n}$, we also have on $(0,+\infty)$ that
 \begin{eqnarray*}
H'(t)&=&2T'(t)T''(t)-2(n-1)\frac{C_{\kappa}(t)}{S_{\kappa}^{3}(t)}T^{2}(t)+\frac{2(n-1)}{S_{\kappa}^{2}(t)}T(t)T'(t)\\
&=&2T'(t)\left[(\phi'(t)-(n-1)\frac{C_{\kappa}(t)}{S_{\kappa}(t)})T'(t)+\frac{n-1}{S_{\kappa}^{2}(t)}T(t)\right]-2(n-1)\frac{S'_{\kappa}(t)}{S_{\kappa}^{3}(t)}T^{2}(r)+ \\
&&\qquad \frac{2(n-1)}{S_{\kappa}^{2}(t)}T(t)T'(t)\\
&=&-\frac{2(n-1)}{S_{\kappa}^{3}(t)}\left[C_{\kappa}(t)S_{\kappa}^{2}(t)(T'(t))^{2}-2S_{\kappa}(t)T(t)T'(t)+S'_{\kappa}(t)T^{2}(t)\right]+2(T'(t))^{2}\phi'(t)\\
&\leq&-\frac{2(n-1)}{S_{\kappa}^{3}(t)}\left(S_{\kappa}(t)T'(t)-T(t)\right)^{2}+2(T'(t))^{2}\phi'(t)\\
&\leq& 0.
\end{eqnarray*}
Thus, $H(t)$ is non-increasing on $(0,+\infty)$.
\end{proof}

The following property is also necessary.

\begin{lemma} \label{lemma2-4}
 (\cite[Lemma 4.4]{CM})
Assume that $\Omega$ is a bounded domain in $\mathbb{R}^n$ (or
$\mathbb{H}^n$) with boundary. If $B_{R}(o)$ is the (geodesic) ball
of radius $R$ and centered at the
 origin $o$
 of
 $\mathbb{R}^{n}$ such that $|\Omega|_{\phi}=|B_{R}(o)|_{\phi}$, and the non-constant
functions $u(t)$ and $v(t)$ defined on $[0,+\infty)$ are
monotonically non-increasing and non-decreasing, respectively, then
\begin{eqnarray*}
&&\int_\Omega v(|x|)d\eta\geq\int_{B_{R}(o)}v(|x|)d\eta,\\
&&\int_\Omega u(|x|)d\eta\leq\int_{B_{R}(o)}u(|x|)d\eta.
\end{eqnarray*}
The equality holds if and only if $\Omega=B_{R}(o)$ (or $\Omega$ is
isometric to $B_{R}(o)$).
\end{lemma}

\section{A proof of Theorems \ref{theo-1} and \ref{theo-2}}  \label{S3}
\renewcommand{\thesection}{\arabic{section}}
\renewcommand{\theequation}{\thesection.\arabic{equation}}
\setcounter{equation}{0}

Now, we have:

 \textbf{\emph{A proof of Theorems \ref{theo-1} and
\ref{theo-2}}}. Set
\begin{eqnarray}  \label{f-add}
f(t)=\left\{
\begin{array}{ll}
T(t),~~&\mathrm{if} ~0\leq t \leq R,\\
T(R),~~&\mathrm{if} ~t> R,
\end{array}
\right.
 \end{eqnarray}
 with $T(t)$ the solution to the system (\ref{2-1}),
and
\begin{eqnarray}
\psi_{i}(\xi)=\frac{x_{i}}{t}, \qquad
P_{i}(t,\xi)=f(t)\psi_{i}(\xi), \qquad i=1,2,\cdots,n.
\end{eqnarray}
By the Brouwer fixed point theorem and using a similar argument to
that of Ashbaugh-Benguriar given in \cite{AB} or Chen given in
\cite{HC}, one can always choose a suitable origin
$o\in\mathrm{hull}(\Omega)$ such that
 \begin{eqnarray*}
\int_{\partial\Omega}P_{i}(t,\xi)\widehat{d\mu}=\int_{\partial\Omega}f(t)\psi_{i}(\xi)\widehat{d\mu}=0,\qquad
i=1,2,\cdots,n,
 \end{eqnarray*}
and
\begin{eqnarray*}
\int_{\partial\Omega}P_{i}(t,\xi)\widehat{d\mu}=\int_{\partial\Omega}f(t)\psi_{i}(\xi)u_{j,\phi}\widehat{d\mu}=0,\qquad
i=2,3,\cdots,n,~j=1,2,\cdots,i-1,
\end{eqnarray*}
where $u_{j,\phi}$ denotes an eigenfunction belonging to the $j$-th
Steklov eigenvalue $\sigma_{j,\phi}(\Omega)$ of the Witten-Laplacian
on the domain $\Omega$.
 Anyway, readers who are not familiar with this process of suitably
 finding the origin $o$ can check the $10^{\mathrm{th}}$-$11^{\mathrm{th}}$ pages of
 our
 previous work
 \cite{CM1} for a similar and detailed explanation.
  Using the
 variational
 characterizations (\ref{cha-1})-(\ref{cha-2}), we have for $1\leq i\leq
 n$ that
  \begin{eqnarray*}
\sigma_{i,\phi}(\Omega)\int_{\partial\Omega}P_{i}^{2}(t,\xi)\widehat{d\mu}\leq\int_{\Omega}|\nabla
P_{i}(t,\xi)|^{2}d\mu
  \end{eqnarray*}
and
\begin{eqnarray*}
\int_{\Omega}|\nabla
P_{i}(t,\xi)|^{2}d\mu=\int_{\Omega}\left[(T'(t))^{2}\psi_{i}^{2}(\xi)+T^{2}(t)|\nabla^{\mathbb{S}^{n-1}}\psi_{i}(\xi)|^{2}S_{\kappa}^{-2}(t)\right]d\mu,
\end{eqnarray*}
which implies
\begin{eqnarray} \label{3-add-1}
\int_{\partial\Omega}P_{i}^{2}(t,\xi)\widehat{d\mu}\leq\frac{1}{\sigma_{i,\phi}(\Omega)}\int_{\Omega}\left[(T'(t))^{2}\psi_{i}^{2}(\xi)+T^{2}(t)|\nabla^{\mathbb{S}^{n-1}}\psi_{i}(\xi)|^{2}S_{\kappa}^{-2}(t)\right]d\mu.
\end{eqnarray}
Summing both sides of (\ref{3-add-1}) over the index $i$ from $1$ to
$n$ results in
 \begin{eqnarray}   \label{3-add-2}
\int_{\partial\Omega}T^{2}(t)\widehat{d\mu}\leq\sum\limits_{i=1}^{n}
\frac{1}{\sigma_{i,\phi}(\Omega)}\int_{\Omega}\left[(T'(t))^{2}\psi_{i}^{2}(\xi)+T^{2}(t)|\nabla^{\mathbb{S}^{n-1}}\psi_{i}(\xi)|^{2}S_{\kappa}^{-2}(t)\right]d\mu.
 \end{eqnarray}
Since
$|\nabla^{\mathbb{S}^{n-1}}\psi_{i}(\xi)|^{2}=1-\psi_{i}^{2}(\xi)$,
we can obtain from (\ref{3-add-2}) that
\begin{eqnarray} \label{3-add-3}
\int_{\partial\Omega}T^{2}(t)\widehat{d\mu}&\leq&\sum\limits_{i=1}^{n}\frac{1}{\sigma_{i,\phi}(\Omega)}\left[\int_{\Omega}(T'(t))^{2}\psi_{i}^{2}(\xi)d\mu+
\int_{\Omega}\frac{T^{2}(t)}{S_{\kappa}^{2}(t)}\left(1-\psi_{i}^{2}(\xi)d\mu\right)\right]\nonumber\\
&=&\sum\limits_{i=1}^{n}\frac{1}{\sigma_{i,\phi}(\Omega)}\left[\int_{\Omega}\frac{T^{2}(t)}{S_{\kappa}^{2}(t)}d\mu+
\int_{\Omega}\left((T'(t))^{2}-\frac{T^{2}(t)}{S_{\kappa}^{2}(t)}\right)\psi_{i}^{2}(\xi)d\mu\right].
\qquad
\end{eqnarray}
Using Lemma \ref{lemma2-3} and a similar calculation to that on page
941 of \cite{HC}, we have
 \begin{eqnarray} \label{3-add-4}
\sum\limits_{i=1}^{n}\frac{1}{\sigma_{i,\phi}(\Omega)}\left((T'(t))^{2}-\frac{T^{2}(t)}{S_{\kappa}^{2}(t)}\right)\psi_{i}^{2}(\xi)\leq
\frac{1}{n-1}\sum_{i=1}^{n-1}\frac{(T'(t))^{2}}{\sigma_{i,\phi}(\Omega)}-\frac{1}{\sigma_{n,\phi}(\Omega)}\frac{T^{2}(t)}{S_{\kappa}^{2}(t)}.
 \end{eqnarray}
Putting (\ref{3-add-4}) into (\ref{3-add-3}) results in
 \begin{eqnarray}  \label{3-add-5}
\int_{\partial\Omega}T^{2}(t)\widehat{d\mu}\leq
\frac{1}{n-1}\sum_{i=1}^{n-1}\frac{1}{\sigma_{i,\phi}(\Omega)}\int_{\Omega}\left((T'(t))^{2}+(n-1)\frac{T^{2}(t)}{S_{\kappa}^{2}(t)}\right)d\mu.
 \end{eqnarray}
Meanwhile, we notice that
\begin{eqnarray} \label{3-add-6}
\int_{\partial\Omega}T^{2}(t)\widehat{d\mu}&\geq&\int_{\partial\Omega}T^{2}(t)e^{-\phi}\langle\nabla t,\vec{\eta}\rangle\widehat{dv}\nonumber\\
&=&\int_{\Omega}\mathrm{div}(T^{2}(t)e^{-\phi}\nabla t)dv\nonumber\\
&=&\int_{\Omega}\left(\nabla(T^{2}(t)e^{-\phi})\nabla t+T^{2}(t)e^{-\phi}\Delta t\right)dv\nonumber\\
&=&\int_{\Omega}\left[2T(t)T'(t)e^{-\phi}+e^{-\phi}(-\phi')T^{2}(t)+\frac{(n-1)C_{\kappa}(t)}{S_{\kappa}(t)}T^{2}(t)e^{-\phi}\right]dv\nonumber\\
&=&\int_{\Omega}G(t)d\mu.
\end{eqnarray}
Combining (\ref{3-add-5}) and (\ref{3-add-6}), and using Theorem
\ref{theorem2-1}, Lemma \ref{lemma2-3}, Lemma \ref{lemma2-4} and a
similar method to that on page 635 of \cite{HFW}, we have
\begin{eqnarray} \label{3-add-7}
\frac{1}{n-1}\sum_{i=1}^{n-1}\frac{1}{\sigma_{i,\phi}(\Omega)}\geq\frac{\int_{\Omega}G(t)d\mu}{\int_{\Omega}H(t)d\mu}
\geq\frac{\int_{B_{R}(o)}G(t)d\mu}{\int_{B_{R}(o)}H(t)d\mu}
=\frac{1}{\sigma_{1,\phi}(B_{R}(o))},
\end{eqnarray}
which implies (\ref{II-1}) or (\ref{II-1-1}) directly. Here, of
course, $G(t)$, $H(t)$ are defined by (\ref{2-8}) and (\ref{2-9})
respectively.

Now, we would like to give the rigidity description for the equality
case in the spectral isoperimetric inequality (\ref{II-1}) or
(\ref{II-1-1}). From (\ref{3-add-7}) and Lemma \ref{lemma2-3}, it is
not hard to see that the radial function $G(t)$ is  nonnegative and
non-decreasing, while the radial function $H(t)$ is  nonnegative and
non-increasing. Moreover, one has
 \begin{eqnarray}  \label{3-add-8}
\int_{\Omega}G(t)d\mu&=&\int_{\Omega\cap
B_{R}(o)}G(t)d\mu+\int_{\Omega\setminus(\Omega\cap
B_{R}(o))}G(t)d\mu \nonumber\\
&\geq& \int_{\Omega\cap
B_{R}(o)}G(t)d\mu+G(R)\int_{\Omega\setminus(\Omega\cap
B_{R}(o))}d\mu,
 \end{eqnarray}
and similarly,
 \begin{eqnarray}  \label{3-add-9}
\int_{B_{R}(o)}G(t)d\mu&=&\int_{\Omega\cap
B_{R}(o)}G(t)d\mu+\int_{B_{R}(o)\setminus(\Omega\cap
B_{R}(o))}G(t)d\mu \nonumber\\
&\leq& \int_{\Omega\cap
B_{R}(o)}G(t)d\mu+G(R)\int_{B_{R}(o)\setminus(\Omega\cap
B_{R}(o))}d\mu.
 \end{eqnarray}
On the other hand, since we have assumed
$|\Omega|_{\phi}=|B_{R}(o)|_{\phi}$, then one has
$$\int_{\Omega\setminus(\Omega\cap
B_{R}(o))}d\mu=\int_{B_{R}(o)\setminus(\Omega\cap B_{R}(o))}d\mu,$$
which, together with (\ref{3-add-8}) and (\ref{3-add-9}), implies
$\int_{\Omega}G(t)d\mu\geq\int_{B_{R}(o)}G(t)d\mu$. Clearly, when
the equality holds, from (\ref{3-add-8}) and (\ref{3-add-9}), it is
not hard to see that
\begin{eqnarray*}
\int_{\Omega\setminus(\Omega\cap
B_{R}(o))}G(t)d\mu=G(R)\int_{B_{R}(o)\setminus(\Omega\cap
B_{R}(o))}d\mu
\end{eqnarray*}
and
 \begin{eqnarray*}
\int_{B_{R}(o)\setminus(\Omega\cap
B_{R}(o))}G(t)d\mu=G(R)\int_{B_{R}(o)\setminus(\Omega\cap
B_{R}(o))}d\mu
 \end{eqnarray*}
hold simultaneously. Then using the monotonicity of the function $G$
and Lemma \ref{lemma2-4}, one knows that in this situation $\Omega$
is the ball $B_{R}(o)\subset\mathbb{R}^n$ (or $\Omega$ is isometric
to the geodesic ball $B_{R}(o)\subset\mathbb{H}^n$). BTW, one can
also make a similar argument on the function $H$ in order to get the
rigidity result for the equality case in the spectral isoperimetric
inequality (\ref{II-1}) or (\ref{II-1-1}). This completes the proof.
\hfill $\square$

\section*{Acknowledgments}
\renewcommand{\thesection}{\arabic{section}}
\renewcommand{\theequation}{\thesection.\arabic{equation}}
\setcounter{equation}{0} \setcounter{maintheorem}{0}

This research was supported in part by the NSF of China (Grant Nos.
11801496 and 11926352), the Fok Ying-Tung Education Foundation
(China), and Hubei Key Laboratory of Applied Mathematics (Hubei
University). The authors would like to thank Dr. Ruifeng Chen for
the useful discussions during the preparation of this paper.


\begin{thebibliography}{9999}


\bibitem{AB} M. S. Ashbaugh, R. D. Benguria, \emph{A sharp bound for the ratio of the first two eigenvalues
of Dirichlet Laplacians and extensions},  Annals of Math. {\bf 135}
(1992) 601--628.

\bibitem{BPR} L. Brasco, G. De Philippis, B. Ruffini, \emph{Spectral optimization
for the Stekloff-Laplacian: the stability issue}, J. Funct. Anal.
{\bf 262} (2012) 4675--4710.

\bibitem{FB} F. Brock, \emph{An isoperimetric inequality for eigenvalues of the Stekloff problem}, Z. Angew. Math. Mech. {\bf 81} (2001)
69--71.

\bibitem{CR} P. Castillon, B. Ruffini, \emph{A spectral characterization of geodesic
balls in non-compact rank one symmetric spaces}, Ann. Sc. Norm.
Super. Pisa Cl. Sci. (5) {\bf 19} (2019) 1359--1388.

\bibitem{HC} H. Chen, \emph{The upper bound of the harmonic mean of the Steklov eigenvalues in curved
spaces}, Bull. Lond. Math. Soc. {\bf 56}(3) (2024) 931--944.

\bibitem{CM} R. F. Chen, J. Mao, \emph{Several isoperimetric inequalities of Dirichlet and Neumann eigenvalues
of the Witten-Laplacian}, submitted and available online at
arXiv:2403.08075v2

\bibitem{CM1} R. F. Chen, J. Mao, \emph{On the Ashbaugh-Benguria type conjecture about lower-order Neumann eigenvalues of the
Witten-Laplacian}, submitted and available online at
arXiv:2403.08070v2


\bibitem{CGGS} B. Colbois, A. Girouard, C. Gordon, D. Sher, \emph{Some recent developments on the Steklov eigenvalue
problem}, Revista Matem\'{a}tica Complutense {\bf 37} (2024) 1--161.

\bibitem{DCM} Y. L. Deng, R. F. Chen, J. Mao, \emph{On the Ashbaugh-Benguria type conjecture about lower-order Neumann eigenvalues of the Witten-Laplacian in curved spaces} (in Chinese),
 submitted.

\bibitem{DDMZ} Y. L. Deng, F. Du, J. Mao, Y. Zhao, \emph{Sharp eigenvalue estimates and related rigidity
theorems}, submitted and available online at arXiv:200313231v2

\bibitem{DMWW} F. Du, J. Mao, Q. L. Wang, C. X. Wu, \emph{Eigenvalue inequalities for the buckling problem of the
drifting Laplacian on Ricci solitons}, J. Differ. Equat. {\bf 260}
(2016) 5533--5564.

\bibitem{JFE} J. F. Escobar, \emph{An isoperimetric inequality and the first Steklov
eigenvalue}, J. Funct. Anal. {\bf 165} (1999)  101--116.


\bibitem{GP} A. Girouard, I. Polterovich, \emph{Spectral geometry of the Steklov problem}, J.
Spectral Theory {\bf 7}(2) (2017) 321--359.

\bibitem{HP} J. Hersch, L. E. Payne, \emph{Extremal principles and isoperimetric
inequalities for some mixed problems of Stekloff's type}, Z. Angew.
Math. Phys. {\bf 19} (1968) 802--817.

\bibitem{KK} N. N. Katz, K. Kondo, \emph{Generalized space forms}, Trans. Am. Math. Soc. {\bf 354} (2002)
2279--2284.

\bibitem{KKKNP} N. Kuznetsov, T. Kulczycki, M. Kwa\'{s}nicki, A. Nazarov,
S. Poborchi, I. Polterovich, B. Siudeja, \emph{The legacy of
Vladimir Andreevich Steklov}, Notices Amer. Math. Soc. {\bf 61}(1)
 (2014) 9--22.

\bibitem{RS} B. Raveendran, G. Santhanam, \emph{Sharp upper bound and a comparison
theorem for the first nonzero Steklov eigenvalue}, J. Ramanujan
Math. Soc. {\bf 29} (2014) 133--154.


 \bibitem{RS} M. Ritor\'{e}, C. Sinestrari, \emph{Geometric flows, isoperimetric
inequalities and hyperbolic geometry}, Mean Curvature Flow and
Isoperimetric Inequalities, Adv. Courses Math. CRM Barcelona, pp.
45--113. Birkh\"{a}user, Basel (2010).

\bibitem{MWS} M. W. Stekloff, \emph{Sur les proble\`{e}mes fondamentaux de la physique
math\'{e}matique}, Ann. Sci. \'{E}cole Norm. Sup. {\bf 19} (1902)
455--490.

\bibitem{HFW} H. F. Weinberger, \emph{An isoperimetric inequality for the $N$-dimensional free membrane problem}, J. Rational Mech. Anal. {\bf 5}
(1956) 633--636.

\bibitem{RW} R. Weinstock, \emph{Inequalities for a classical eigenvalue problem},
J. Rational Mech. Anal. {\bf 3} (1954) 745--753.


\bibitem{SJZ} S. J. Zhang, \emph{Brock-type isoperimetric inequality for Steklov eigenvalues of the
Witten-Laplacian} (in Chinese), Master Degree's Thesis, Hubei
University (2025).




\end{thebibliography}
\end{document}